\documentclass{article}
\usepackage[final]{graphicx}
\usepackage{psfrag}
\usepackage{amsmath,amsfonts,amssymb,amsxtra,subeqnarray,bm}

\newlength{\extramargin}
\newcommand{\K}{\ensuremath{\mathcal{K}}}
\newcommand{\Real}{\ensuremath{{\mathbb{R}}}}

\newcommand{\Complex}{\ensuremath{{\mathbb{C}}}}

\newcommand{\V}{\ensuremath{\mathcal V}}

\newcommand{\M}{\ensuremath{\mathcal M}}

\newcommand{\E}{\ensuremath{\mathcal E}}

\newcommand{\G}{\ensuremath{\mathcal G}}

\newcommand{\B}{\ensuremath{\mathcal B}}

\newcommand{\one}{\ensuremath{{\mathbf{1}}}}

\newtheorem{theorem}{Theorem}

\newtheorem{corollary}{Corollary}

\newtheorem{lemma}{Lemma}

\newtheorem{definition}{Definition}

\newenvironment{proof}{\noindent {\bf Proof.}}{\hfill \hspace*{1pt}\hfill$\blacksquare$}

\begin{document}
\title{Harmonic synchronization under all three types of coupling: position, velocity, and acceleration}
\author{S. Emre Tuna\footnote{The author is with Department of
Electrical and Electronics Engineering, Middle East Technical
University, 06800 Ankara, Turkey. Email: {\tt etuna@metu.edu.tr}}}
\maketitle

\begin{abstract}
Synchronization of identical harmonic oscillators interconnected via
position, velocity, and acceleration couplings is studied. How to
construct a complex Laplacian matrix representing the overall
coupling is presented. It is shown that the oscillators
asymptotically synchronize if and only if this matrix has a single
eigenvalue on the imaginary axis. This result generalizes some of
the known spectral tests for synchronization. Some simpler Laplacian
constructions are also proved to work provided that certain
structural conditions are satisfied by the coupling graphs.
\end{abstract}

\section{Introduction}

If a group of identical harmonic oscillators $m_{0}{\ddot
x}_{i}+k_{0}x_{i}=0$ (where $m_{0},\,k_{0}>0$ and
$x_{1},\,x_{2},\,\ldots,\,x_{q}\in\Real$) are coupled through their
relative {\em velocities} to form a network
\begin{eqnarray}\label{eqn:R}
m_{0}{\ddot x}_{i}+k_{0}x_{i}+\sum_{j=1}^{q}b_{ij}({\dot
x}_{i}-{\dot x}_{j})=0
\end{eqnarray}
(where $b_{ji}=b_{ij}\geq 0$ and $b_{ii}=0$) they sometimes display
a remarkable behavior: synchronization, i.e.,
$|x_{i}(t)-x_{j}(t)|\to 0$ for all $(i,\,j)$ as $t\to\infty$. When
they shall synchronize (or fail to do so) is now well known. All one
has to do is check whether the graph $\B$ that the coupling
$(b_{ij})_{i,j=1}^{q}$ gives rise to\footnote{The graph $\B$ has $q$
nodes and there is an edge between $i$th and $j$th nodes if
$b_{ij}>0$.} is connected\footnote{See, e.g., \cite{diestel97} for
the definition of {\em connected graph}.} (or not). There is also an
equivalent, yet more technical, test to determine synchronization.
It employs the graph Laplacian
\begin{eqnarray*}
B =
\left[\begin{array}{cccc}\sum_{j}b_{1j}&-b_{12}&\cdots&-b_{1q}\\
-b_{21}&\sum_{j}b_{2j}&\cdots&-b_{2q}\\
\vdots&\vdots&\ddots&\vdots\\
-b_{q1}&-b_{q2}&\cdots&\sum_{j}b_{qj}\end{array}\right]=:{\rm
lap}\,(b_{ij})_{i,j=1}^{q}
\end{eqnarray*}
and makes a special case of \cite[Thm.~3.1]{ren08}: \vspace{0.12in}

\noindent{\bf Test~1.} The oscillators~\eqref{eqn:R} synchronize if
and only if $\lambda_{2}(B)>0$.\footnote{$\lambda_{i}(A)$ denotes
the $i$th eigenvalue of $A\in\Complex^{q\times q}$ with respect to
the ordering ${\rm Re}\,\lambda_{1}(A)\leq{\rm
Re}\,\lambda_{2}(A)\leq\cdots\leq{\rm Re}\,\lambda_{q}(A)$.}
\vspace{0.12in}

There are instances in the physical world where the {\em position}
coupling also plays a role in shaping the overall interconnection
among the oscillators \cite{tuna17}. This has motivated the
extension of the model~\eqref{eqn:R} to
\begin{eqnarray}\label{eqn:RL}
m_{0}{\ddot x}_{i}+k_{0}x_{i}+\sum_{j=1}^{q}b_{ij}({\dot
x}_{i}-{\dot x}_{j})+\sum_{j=1}^{q}k_{ij}(x_{i}-x_{j})=0
\end{eqnarray}
where $k_{ji}=k_{ij}\geq 0$ and $k_{ii}=0$. Let $\K$ be the graph
associated to the position coupling and $K={\rm
lap}\,(k_{ij})_{i,j=1}^{q}$ denote its Laplacian. Note that now we
have a pair of graphs $(\B,\,\K)$, as opposed to a single one,
describing the overall coupling. Unlike its simpler
version~\eqref{eqn:R} this more interesting setup~\eqref{eqn:RL}
does not admit a nontechnical condition, where synchronization can
be studied solely via graph connectivity. In particular, for the
synchronization of the oscillators~\eqref{eqn:RL}, neither it is
necessary that both $\B$ and $\K$ are separately connected nor it is
sufficient that their union $\B\cup\K$ is. Even though the
connectivity condition does not yield a straightforward extension,
it turns out that Test~1 does. In a recent work
\cite[Cor.~6]{tuna19} it has been shown that

\vspace{0.12in}

\noindent{\bf Test~2.} The oscillators~\eqref{eqn:RL} synchronize if
and only if ${\rm Re}\,\lambda_{2}(B+jK)>0$. \vspace{0.12in}

If one continues to walk in the direction of generalization that
took us from the solely velocity-coupled network~\eqref{eqn:R} to
both position- and velocity-coupled one~\eqref{eqn:RL}, the obvious
next stop is the setup where the {\em acceleration} coupling is also
present. Namely,
\begin{eqnarray}\label{eqn:inertial}
m_{0}{\ddot x}_{i}+k_{0}x_{i}+\sum_{j=1}^{q}m_{ij}({\ddot
x}_{i}-{\ddot x}_{j})+\sum_{j=1}^{q}b_{ij}({\dot x}_{i}-{\dot
x}_{j})+\sum_{j=1}^{q}k_{ij}(x_{i}-x_{j})=0
\end{eqnarray}
where $m_{ji}=m_{ij}\geq 0$ and $m_{ii}=0$. In accordance with our
previous notation we introduce the Laplacian $M={\rm
lap}\,(m_{ij})_{i,j=1}^{q}$ whose graph is denoted by $\M$. The
motivation for studying this general coupling scheme is not purely
theoretical; certain electrical oscillator networks under RLC-type
coupling indeed obey the dynamics~\eqref{eqn:inertial}.
\begin{figure}[h]
\begin{center}
\includegraphics[scale=0.45]{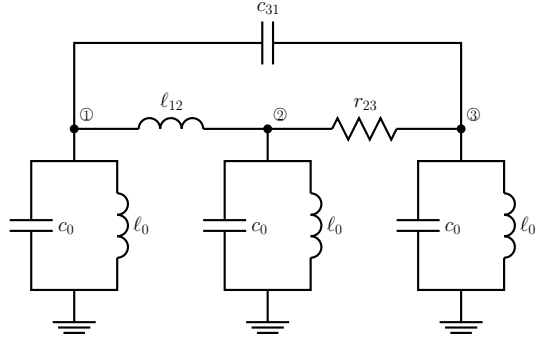}
\caption{A network of LC-tanks under RLC
coupling.}\label{fig:tanks3}
\end{center}
\end{figure}
Consider, for instance, the linear time-invariant (LTI) network of
$q=3$ coupled LC-tanks shown in Fig.~\ref{fig:tanks3}; where $c_{0}$
and $\ell_{0}$ are, respectively, the capacitance and the inductance
of the individual oscillators, $\ell_{12}$ is the inductance of the
inductor connecting the nodes \textcircled{\raisebox{-1pt}{1}} and
\textcircled{\raisebox{-1pt}{2}}, $r_{23}$ is the resistance of the
resistor connecting the nodes \textcircled{\raisebox{-1pt}{2}} and
\textcircled{\raisebox{-1pt}{3}}, and $c_{31}$ is the capacitance of
the capacitor connecting the nodes \textcircled{\raisebox{-1pt}{3}}
and \textcircled{\raisebox{-1pt}{1}}. Letting $x_{i}$ be the $i$th
node voltage, the dynamics of this simple example circuitry read
\begin{eqnarray*}
c_{0}{\ddot x}_{1}+\ell_{0}^{-1}x_{1}+c_{31}({\ddot x}_{1}-{\ddot
x}_{3})+\ell_{12}^{-1}(x_{1}-x_{2})&=&0\\
c_{0}{\ddot x}_{2}+\ell_{0}^{-1}x_{2}+r_{23}^{-1}({\dot x}_{2}-{\dot
x}_{3})+\ell_{12}^{-1}(x_{2}-x_{1})&=&0\\
c_{0}{\ddot x}_{3}+\ell_{0}^{-1}x_{3}+c_{31}({\ddot x}_{3}-{\ddot
x}_{1})+r_{23}^{-1}({\dot x}_{3}-{\dot x}_{2})&=&0\,.
\end{eqnarray*}

The clear-cut goal we intend to achieve in this note is to find the
missing ``Test 3'' which is supposed to tell us when the
oscillators~\eqref{eqn:inertial} synchronize. More formally, we will
investigate the conditions under which an LTI network of identical
harmonic oscillators interconnected by  inertial (acceleration),
dissipative (velocity), and restorative (position) couplings
asymptotically reach synchronization. Even though the nature of the
coupling in our setup~\eqref{eqn:inertial} is quite primitive (in
the sense that it is fixed and symmetric) it still is capable of
accounting for a rich variety of cases, for the overall
interconnection exercises itself through three different mediums
(acceleration, velocity, position) over three independent graphs
$(\M,\B,\,\K)$. We believe that the problem we study here is novel.
To the best of our knowledge, the collective behavior of harmonic
oscillators connected via these three separate coupling graphs has
not been studied before, despite the fact that there are real-world
networks (e.g., Fig.~1) that would benefit from such an analysis. In
particular, the possible effects of relative acceleration coupling
on the evolution of simple oscillator networks (being a fertile
subject of investigation notwithstanding) are yet unbeknownst to the
lively literature on harmonic synchronization, which we review next.

The literature on synchronization of coupled harmonic oscillators
has reached a certain maturity in the last decade. Most of the
initial results concerned pure velocity coupling; works on position
coupling appearing only later. One of the first comprehensive
analyses of harmonic oscillators within the synchronization
framework can be found in \cite{ren08}, where Ren considers
time-varying oscillator dynamics under time-varying and asymmetrical
velocity coupling. This work later enjoyed certain variations and
generalizations. For instance, a type of coupling that becomes
inactive when the distance between oscillators exceeds a threshold
is studied in \cite{su09}. Nonlinearly-coupled harmonic oscillators
are analyzed in \cite{cai10}, where an averaging technique is
employed to establish synchronization. Among many other articles
studying velocity-coupled harmonic oscillators are \cite{zhou12},
where the information exchange between units takes place in an
impulsive fashion; \cite{zhang12,sun14}, where sampled-data
approaches are proposed to study synchronization; and \cite{song19},
where both delayed measurements and negative coupling weights are
allowed. The early investigations on the effect and utility of
position coupling seem to go as far back as the work \cite{zhang13},
where position coupling is considered together with velocity
coupling, but not independently, in the sense that they share the
same single Laplacian matrix. The works that succeeded
\cite{zhang13} can be classified into two groups. One group removed
velocity coupling from the picture altogether and allowed relative
position measurements only, while the other group allowed in their
setup both position and velocity couplings, where each has its own
separate Laplacian matrix. To the first group belong, for instance,
\cite{song16,zhang18}, where synchronization is established via
sampled-data strategies. A generalization to heterogeneous harmonic
oscillator networks is later presented in \cite{song19b}. Also
related to the first group is \cite{liu19}, where bipartite
consensus problem is considered under sampled position data. The
second group contains the work \cite{tuna17}, where an
observability-like condition for synchronization is presented in
terms of the pair of Laplacians describing the overall
interconnection; and \cite{wang19}, where practical stochastic
synchronization is studied under position and velocity couplings.

The remainder of the paper is organized as follows. In
Section~\ref{sec:second} we present a complex Laplacian matrix
construction out of the network parameters
$(M,\,B,\,K,\,m_{0},\,k_{0})$ and an associated eigenvalue test
(that generalizes Tests~1 and 2) to determine whether the array of
oscillators~\eqref{eqn:inertial} synchronize. We also provide a
numerical example to emphasize the fact that synchronous behavior
(or its absence) does depend on the individual oscillator parameters
$(m_{0},\,k_{0})$; a peculiarity that the simpler
networks~\eqref{eqn:R} and \eqref{eqn:RL} do not suffer from. Then,
in Section~\ref{sec:simpler}, we bring forth some structural
conditions on the coupling graphs $\M,\,\B,\,\K$ under which the
eigenvalue test presented in Section~\ref{sec:second} takes much
simpler forms.

\section{The second eigenvalue}\label{sec:second}

Consider the network of $q$ coupled harmonic
oscillators~\eqref{eqn:inertial}. When these units will eventually
oscillate in unison is what we aim to find out here. For our
purpose, we focus on the implications of the spectral properties of
a $q$-by-$q$ complex Laplacian matrix (yet to be constructed) on
synchronization.

\begin{definition}
The oscillators~\eqref{eqn:inertial} are said to {\em synchronize}
if the solutions satisfy $|x_{i}(t)-x_{j}(t)|\to 0$ as $t\to\infty$
for all $(i,\,j)$ and all initial conditions.
\end{definition}

The identity matrix is denoted by $I\in\Real^{q\times q}$ and the
vector of all ones by $\one_{q}\in\Real^{q}$. By construction the
(previously defined) Laplacian matrices $M,\,B,\,K\in\Real^{q\times
q}$ are all symmetric positive semidefinite and the null space of
each contains the vector $\one_{q}$. By letting $x=[x_{1}\ x_{2}\
\cdots\ x_{q}]^{T}\in\Real^{q}$ we can rewrite the
dynamics~\eqref{eqn:inertial} as
\begin{eqnarray}\label{eqn:x}
(M+m_{0}I){\ddot x}+B{\dot x}+(K+k_{0}I)x=0\,.
\end{eqnarray}
We observe that every solution $x(t)$ of \eqref{eqn:x} is bounded.
This fact can be established using the function
\begin{eqnarray}\label{eqn:V}
V=\frac{1}{2}x^{T}(K+k_{0}I)x+\frac{1}{2}{\dot x}^{T}(M+m_{0}I){\dot
x}
\end{eqnarray}
which is nonnegative because both $(K+k_{0}I)=:K_{\rm a}$ and
$(M+m_{0}I)=:M_{\rm a}$ are positive definite matrices. Combining
\eqref{eqn:x} and \eqref{eqn:V} yields ${\dot V}=-{\dot x}^{T}B{\dot
x}$. Since $B$ is positive semidefinite we have ${\dot V}\leq 0$.
Therefore $V(t)\leq V(0)$ for all $t\geq 0$, which at once implies
the boundedness of $x(t)$. Now, being produced by an LTI system,
$x(t)$ can be written as a finite sum
\begin{eqnarray}\label{eqn:sum}
x(t)=\sum_{k}{\rm Re}\,(e^{\lambda_{k}t}p_{k}(t))
\end{eqnarray}
where $\lambda_{k}\in\Complex$ are distinct and $p_{k}(t)$ are
polynomials with vector coefficients. In the light of boundedness we
can then assert that ${\rm Re}\,\lambda_{k}\leq 0$ for all $k$ and
if ${\rm Re}\,\lambda_{k}=0$ for some $k$ then the corresponding
polynomial $p_{k}(t)$ must necessarily be of degree zero, i.e., a
constant vector. Suppose now the oscillators~\eqref{eqn:inertial}
fail to synchronize. This implies that there exists a
solution~\eqref{eqn:sum} where the sum contains an index $k$ for
which $\lambda_{k}=j\omega$ and $p_{k}(t)\equiv \xi$ with
$\omega\in\Real_{>0}$ and $\xi\in\Complex^{q}\setminus{\rm
span}\,\{\one_{q}\}$. Because the function $t\mapsto e^{j\omega
t}\xi$ has to satisfy \eqref{eqn:x}, we have
\begin{eqnarray}\label{eqn:cherry}
(K_{\rm a}-\omega^{2}M_{\rm a})\xi+j\omega B\xi=0\,.
\end{eqnarray}
Multiplying the above equation from left by $\xi^{*}$ yields
\begin{eqnarray}\label{eqn:plum}
\xi^{*}K_{\rm a}\xi-\omega^{2}\xi^{*}M_{\rm
a}\xi+j\omega\xi^{*}B\xi=0\,.
\end{eqnarray}
Note that the terms $\xi^{*}K_{\rm a}\xi,\,\xi^{*}M_{\rm
a}\xi,\,\xi^{*}B\xi$ are all real because $K_{\rm a},\,M_{\rm
a},\,B\geq 0$. Therefore \eqref{eqn:plum} implies $\xi^{*}B\xi=0$.
Since $B$ is symmetric positive semidefinite this means $B\xi=0$,
whence $(K_{\rm a}-\omega^{2}M_{\rm a})\xi=0$ by \eqref{eqn:cherry}.
To summarize, if the oscillators~\eqref{eqn:inertial} do not
synchronize then there exist a real number $\omega>0$ and a vector
$\xi\notin{\rm span}\,\{\one_{q}\}$ such that
\begin{eqnarray}\label{eqn:pine}
\left[\begin{array}{c}(K+k_{0}I)-\omega^{2}(M+m_{0}I)\\
B\end{array}\right]\xi=0\,.
\end{eqnarray}
It is not difficult to see that the steps we have taken are
reversible. That is, if we can find a real number $\omega>0$ and a
vector $\xi\notin{\rm span}\,\{\one_{q}\}$ satisfying
\eqref{eqn:pine} then we can construct the function $t\mapsto {\rm
Re}\,(e^{j\omega t}\xi)$ which solves \eqref{eqn:x} thanks to
\eqref{eqn:pine}. And this cannot a synchronous solution because
$\xi\notin{\rm span}\,\{\one_{q}\}$. We have therefore established:

\begin{lemma}\label{lem:one}
The following are equivalent.
\begin{enumerate}
\item The oscillators~\eqref{eqn:inertial} do not synchronize.
\item There exist $\omega\in\Real_{>0}$ and
$\xi\in\Complex^{q}\setminus{\rm span}\,\{\one_{q}\}$ satisfying
\eqref{eqn:pine}.
\end{enumerate}
\end{lemma}

The above lemma can be a useful test for synchronization, but it is
worthwhile to search for a simpler way to determine when the
oscillators synchronize. We now present the following alternative.

\begin{theorem}\label{thm:test3}
The oscillators~\eqref{eqn:inertial} synchronize if and only if
${\rm Re}\,\lambda_{2}(\Lambda)>0$ where
\begin{eqnarray}\label{eqn:test3}
\Lambda:=(M+m_{0}I)^{-1/2}(B+j(K+k_{0}I))(M+m_{0}I)^{-1/2}-j\frac{k_{0}}{m_{0}}I
\end{eqnarray}
is the complex Laplacian representing the network.
\end{theorem}

Note that when there is no inertial coupling (i.e., $M=0$) we have
$\Lambda=m_{0}^{-1}(B+jK)$ and the condition for synchronization
presented in Theorem~\ref{thm:test3} can be written as ${\rm
Re}\,\lambda_{2}(m_{0}^{-1}(B+jK))>0$ which clearly is equivalent to
${\rm Re}\,\lambda_{2}(B+jK)>0$. Furthermore, if there is only
dissipative coupling (i.e., both $M=0$ and $K=0$) the condition
further reduces to ${\rm Re}\,\lambda_{2}(m_{0}^{-1}B)>0$ which is
equivalent to $\lambda_{2}(B)>0$ since $B$ is real and symmetric.
Therefore Theorem~\ref{thm:test3} generalizes the Tests~1 and 2
mentioned earlier in the paper. Note however that this
generalization has one qualitative aspect which its corollaries do
not manifest: it appears to depend not only on the coupling
$(M,\,B,\,K)$ but also on the individual oscillator parameters
$(m_{0},\,k_{0})$. Is this a superficial dependence? If not, there
should exist a coupling $(M,\,B,\,K)$ for which one can find two
pairs $(m_{0}',\,k_{0}')$ and $(m_{0}'',\,k_{0}'')$ such that the
array of oscillators~\eqref{eqn:x} described by
$(M,\,B,\,K,\,m_{0}',\,k_{0}')$ synchronize whereas the other set of
parameters $(M,\,B,\,K,\,m_{0}'',\,k_{0}'')$ produces asynchronous
solutions. It turns out that such couplings are not difficult to
come by. (Hence the answer to our question is {\em no}.) We provide
an example below.

\begin{figure}[h]
\begin{center}
\includegraphics[scale=0.45]{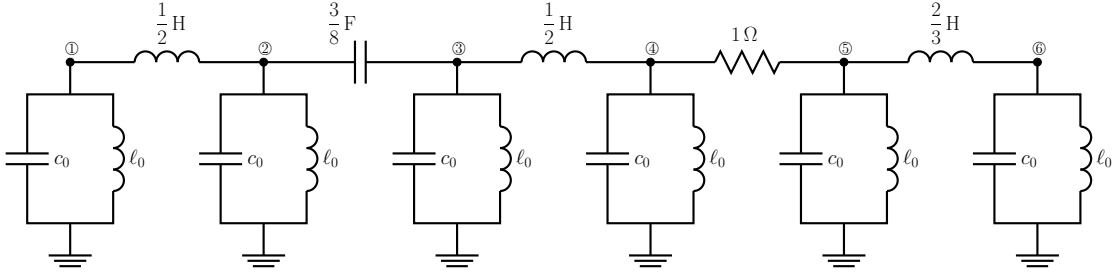}
\caption{A network of LC-tanks coupled via LTI capacitors,
inductors, and resistors. The oscillators synchronize for
$(c_{0},\,\ell_{0})=(2,\,\frac{1}{2})$ but not for
$(c_{0},\,\ell_{0})=(1,\,1)$.}\label{fig:tanks6}
\end{center}
\end{figure}

Consider the network of six coupled LC-tanks shown in
Fig.~\ref{fig:tanks6}. Letting $x=[x_{1}\ x_{2}\ \cdots\ x_{6}]^{T}$
denote the node voltage vector and setting $m_{0}=c_{0}$ and
$k_{0}=\ell_{0}^{-1}$ the dynamics of this network obey
\eqref{eqn:x} by the following coupling matrices
\begin{eqnarray*}
M = \left[\arraycolsep=3.3pt\def\arraystretch{1.2}
\begin{array}{rrrrrr}
0&0&0&0&0&0\\
0&\frac{3}{8}&-\frac{3}{8}&0&0&0\\
0&-\frac{3}{8}&\frac{3}{8}&0&0&0\\
0&0&0&0&0&0\\
0&0&0&0&0&0\\
0&0&0&0&0&0
\end{array}\right],\quad
B = \left[ \arraycolsep=3.3pt\def\arraystretch{1.2}
\begin{array}{rrrrrr}
0&0&0&0&0&0\\
0&0&0&0&0&0\\
0&0&0&0&0&0\\
0&0&0&1&-1&0\\
0&0&0&-1&1&0\\
0&0&0&0&0&0
\end{array}\right],\quad
K = \left[ \arraycolsep=3.3pt\def\arraystretch{1.2}
\begin{array}{rrrrrr}
2&-2&0&0&0&0\\
-2&2&0&0&0&0\\
0&0&2&-2&0&0\\
0&0&-2&2&0&0\\
0&0&0&0&\frac{3}{2}&-\frac{3}{2}\\
0&0&0&0&-\frac{3}{2}&\frac{3}{2}
\end{array}\right].
\end{eqnarray*}
Let us first study this circuit under the oscillator parameters
$c_{0}=2$F and $\ell_{0}=\frac{1}{2}$H, yielding
$(m_{0},\,k_{0})=(2,\,2)$ and $\omega_{0}=1$rad/sec (where
$\omega_{0}=\sqrt{k_{0}/m_{0}}$ is the frequency of uncoupled
oscillations). The eigenvalues of the associated
Laplacian~\eqref{eqn:test3} can be computed as
$\lambda_{1}=0,\,\lambda_{2}=0.0078-j0.1409,\,\lambda_{3}=0.0088+j1.5747
,\,\lambda_{4}=0.0434+j1.9338
,\,\lambda_{5}=0.4452+j0.1386,\,\lambda_{6}=0.4947+j1.4484$. Since
the second eigenvalue is on the open right half-plane the
oscillators synchronize by Theorem~\ref{thm:test3}. Consider once
again the array in Fig.~\ref{fig:tanks6}, this time with $c_{0}=1$F
and $\ell_{0}=1$H, yielding $(m_{0},\,k_{0})=(1,\,1)$ and
$\omega_{0}=1$rad/sec. Note that in this second case even though we
changed the oscillator parameters, the frequency of uncoupled
oscillations is still the same. Despite this sameness however the
eigenvalues of the Laplacian
($\lambda_{1}=0,\,\lambda_{2}=j3,\,\lambda_{3}=0.0107-j0.2436,\,\lambda_{4}=0.0996
+j3.8647,\,\lambda_{5}=0.8666+j0.2996,\,\lambda_{6}=1.0230+j2.7936$)
tell us that the oscillators will fail to synchronize because the
condition ${\rm Re}\,\lambda_{2}(\Lambda)>0$ no longer holds. We end
this section by the proof of Theorem~\ref{thm:test3}.

\vspace{0.12in}

\noindent{\bf Proof of Theorem~\ref{thm:test3}.} Recall the
shorthand notation $M_{\rm a}=M+m_{0}I$, $K_{\rm a}=K+k_{0}I$, and
$\omega_{0}^{2}=k_{0}/m_{0}$ we introduced earlier. We first
establish some properties of the Laplacian~\eqref{eqn:test3}. Since
the real matrices $M,\,K$ are symmetric positive semidefinite, the
augmented matrices $M_{\rm a},\,K_{\rm a}$ are symmetric positive
definite. Furthermore, for an eigenvalue $\alpha\in\Real$ and the
corresponding eigenvector $y\in\Complex^{q}$ satisfying $My=\alpha
y$ it is clear that we can write $M_{\rm
a}^{\sigma}y=(\alpha+m_{0})^{\sigma}y$ for any power
$\sigma\in\Real$. Likewise, $Ky=\alpha y$ implies $K_{\rm
a}^{\sigma}y=(\alpha+k_{0})^{\sigma}y$. This at once yields
$\Lambda\one_{q}=0$ since $M\one_{q}=B\one_{q}=K\one_{q}=0$. Let
$D:=M_{\rm a}^{-1/2}BM_{\rm a}^{-1/2}$ and $R:=M_{\rm
a}^{-1/2}K_{\rm a}M_{\rm a}^{-1/2}$. Note that $D,\,R\geq 0$. We now
show that $\Lambda$ can have no eigenvalue on the open left
half-plane. Let $\lambda\in\Complex$ be an eigenvalue of $\Lambda$
and $u\in\Complex^{q}$ be the corresponding unit eigenvector. That
is, $\Lambda u=\lambda u$ and $u^{*}u=1$. We can write
\begin{eqnarray*}
\lambda = u^{*}\Lambda
u=u^{*}(D+jR-j\omega_{0}^{2}I)u=u^{*}Du+j(u^{*}Ru-\omega_{0}^{2})
\end{eqnarray*}
whence follows that ${\rm Re}\,\lambda=u^{*}Du\geq 0$. Combining
this with the fact that $\Lambda$ has an eigenvalue at the origin
(recall $\Lambda\one_{q}=0$) allows us to write, without loss of
generality, $\lambda_{1}(\Lambda)=0.$

Suppose now ${\rm Re}\,\lambda_{2}(\Lambda)\leq0$. Since ${\rm
Re}\,\lambda_{i}(\Lambda)\geq 0$ for all $i$ we have to have
$\lambda_{2}(\Lambda)=j\mu$ for some $\mu\in\Real$. This implies the
existence of an eigenvector $\eta\notin{\rm span}\,\{\one_{q}\}$
satisfying $\Lambda\eta=j\mu\eta$. This is obvious if $\mu\neq 0$.
It is still true when $\mu=0$, i.e., when the eigenvalue at the
origin is repeated. To see that suppose otherwise, i.e., $\one_{q}$
were the sole eigenvector corresponding to the eigenvalue at the
origin. Since the eigenvalue at the origin is repeated we then would
have to have a generalized eigenvector $\eta_{\rm g}$ satisfying
$\Lambda\eta_{\rm g}=\one_{q}$. This however would produce the
contradiction
\begin{eqnarray*}
q=\one_{q}^{T}\one_{q}=\one_{q}^{T}\Lambda\eta_{\rm
g}=(\Lambda\one_{q})^{T}\eta_{\rm g}=0
\end{eqnarray*}
due to the symmetry $\Lambda^{T}=\Lambda$. Now, without loss of
generality let $\eta^{*}\eta=1$. We can write
\begin{eqnarray*}
j\mu=\eta^{*}\Lambda\eta=\eta^{*}D\eta+j(\eta^{*}R\eta-\omega_{0}^{2})
\end{eqnarray*}
which tells us $\eta^{*}D\eta=0$. Consequently, since the real
matrix $D$ is symmetric positive semidefinite, we have $D\eta=0$.
Recalling $D=M_{\rm a}^{-1/2}BM_{\rm a}^{-1/2}$ and defining
$\xi:=M_{\rm a}^{-1/2}\eta$ we can then assert
\begin{eqnarray}\label{eqn:Bxi}
B\xi=0
\end{eqnarray}
because $M_{\rm a}^{-1/2}$ is nonsingular. Observe that
$\xi\notin{\rm span}\,\{\one_{q}\}$. This follows from the fact that
$\one_{q}$ is an eigenvector of $M$ and, consequently, of $M_{\rm
a}^{1/2}$. That is, $M_{\rm a}^{1/2}\one_{q}\in{\rm
span}\,\{\one_{q}\}$. Hence, if $\xi$ did belong to ${\rm
span}\,\{\one_{q}\}$ then we would have $\eta=M_{\rm
a}^{1/2}\xi\in{\rm span}\,\{\one_{q}\}$. But this contradicts
$\eta\notin{\rm span}\,\{\one_{q}\}$. Combining $D\eta=0$ and
$\Lambda\eta=j\mu\eta$ we obtain $R\eta-\omega_{0}^{2}\eta=\mu\eta$
yielding $R\eta=(\omega_{0}^{2}+\mu)\eta$. Since $R$ is symmetric
positive definite all its eigenvalues are real and positive. This
means $\omega=\sqrt{\omega_{0}^{2}+\mu}>0$
satisfies
\begin{eqnarray*} 0=(R-\omega^{2}I)\eta=(M_{\rm
a}^{-1/2}K_{\rm a}M_{\rm a}^{-1/2}-\omega^{2}I)\eta=M_{\rm
a}^{-1/2}(K_{\rm a}-\omega^{2}M_{\rm a})M_{\rm a}^{-1/2}\eta
\end{eqnarray*}
which lets us see
\begin{eqnarray}\label{eqn:Kaw2Ma}
((K+k_{0}I)-\omega^{2}(M+m_{0}I))\xi=0\,.
\end{eqnarray}
Combining \eqref{eqn:Bxi}, \eqref{eqn:Kaw2Ma}, and
Lemma~\ref{lem:one} we finally establish that the
oscillators~\eqref{eqn:inertial} do not synchronize.

We now show the other direction. Suppose the
oscillators~\eqref{eqn:inertial} do not synchronize. Then by
Lemma~\ref{lem:one} there exist $\omega>0$ and $\xi\notin{\rm
span}\,\{\one_{q}\}$ satisfying \eqref{eqn:Bxi} and
\eqref{eqn:Kaw2Ma}. Let $\mu=\omega^{2}-\omega_{0}^{2}$ and
$\eta=M_{\rm a}^{1/2}\xi$. By retracing the steps we have taken in
the first part of the proof we can easily reach
$\Lambda\eta=j\mu\eta$ as well as establishing $\eta\notin{\rm
span}\,\{\one_{q}\}$. This means (because $\Lambda\one_{q}=0$) that
the Laplacian $\Lambda$ has at least two eigenvalues on the
imaginary axis. Hence we conclude that ${\rm
Re}\,\lambda_{2}(\Lambda)=0$ since all the eigenvalues of $\Lambda$
are on the closed right half-plane.\hfill\null\hfill$\blacksquare$

\section{Simpler characterizations under structural
conditions}\label{sec:simpler}

In the previous section we have seen that for a given coupling
$(M,\,B,\,K)$ whether the oscillators~\eqref{eqn:inertial}
synchronize or not depends in general on the individual oscillator
parameters $(m_{0},\,k_{0})$ as well. In this section we investigate
structural conditions on the coupling under which synchronization
depends solely on the triple $(M,\,B,\,K)$. To this end, we need
some notation first. A graph $\G$ is a pair $(\V,\,\E)$ where
$\V=\{v_{1},\,v_{2},\,\ldots,\,v_{q}\}$ is the set of vertices
(nodes) and the set $\E$ contains some (unordered) pairs
$(v_{i},\,v_{j})$ with $i\neq j$, called edges. The set of vertices
incident to an edge is denoted by ${\rm ver}\,\E\subset\V$. That is,
${\rm ver}\,\E=\{v_{i}:(v_{i},\,v_{j})\in\E\}$. We define the edge
set describing the inertial coupling as $\E_{\rm
m}=\{(v_{i},\,v_{j}):m_{ij}>0\}$. The sets $\E_{\rm b}$ and $\E_{\rm
k}$ are defined, mutatis mutandis, for the dissipative and
restorative couplings, respectively. Therefore the three graphs
(introduced earlier) describing the coupling in the
network~\eqref{eqn:inertial} can be written as $\M=(\V,\,\E_{\rm
m})$, $\B=(\V,\,\E_{\rm b})$, and $\K=(\V,\,\E_{\rm k})$.

\begin{theorem}\label{thm:B}
The oscillators~\eqref{eqn:inertial} synchronize if the graph $\B$
is connected.
\end{theorem}

\begin{proof}
That $\B$ is connected means its Laplacian $B$ satisfies ${\rm
null}\,B={\rm span}\,\{\one_{q}\}$; see, for instance,
\cite[Lem.~3.1]{ren08}. The result then follows by
Lemma~\ref{lem:one}.
\end{proof}

\vspace{0.12in}

For many applications, connectedness of the dissipative coupling
graph could be too conservative an assumption. We now attempt to
relax this requirement utilizing {\em isolation}, by which we mean
the following. Two graphs (defined over the same vertex set) are
isolated when their edges do not touch one another. More formally:

\begin{definition}
The graphs $\M=(\V,\,\E_{\rm m})$ and $\K=(\V,\,\E_{\rm k})$ are
said to be {\em edge-isolated} if ${\rm ver}\,\E_{\rm m}\cap{\rm
ver}\,\E_{\rm k}=\emptyset$.
\end{definition}

\begin{theorem}\label{thm:MK}
Suppose the graphs $\M$ and $\K$ are edge-isolated. Then the
oscillators~\eqref{eqn:inertial} synchronize if and only if ${\rm
Re}\,\lambda_{2}(B+j(K-M))>0$.
\end{theorem}

An immediate implication of Theorem~\ref{thm:MK} concerning the type
of electrical networks we considered earlier in the paper is the
following. If the coupling network is such that there is not a
single node where the terminals of a capacitive connector and an
inductive connector meet then whether the oscillators synchronize or
not does not depend on the individual oscillator parameters
$(c_{0},\,\ell_{0})$. Note also that Test~2 follows from
Theorem~\ref{thm:MK} as a special case. In addition,
Theorem~\ref{thm:MK} produces the following sister test.

\begin{corollary}
The coupled oscillators
\begin{eqnarray*}
m_{0}{\ddot x}_{i}+k_{0}x_{i}+\sum_{j=1}^{q}m_{ij}({\ddot
x}_{i}-{\ddot x}_{j})+\sum_{j=1}^{q}b_{ij}({\dot x}_{i}-{\dot
x}_{j})=0
\end{eqnarray*}
synchronize if and only if ${\rm Re}\,\lambda_{2}(B-jM)>0$.
\end{corollary}

We need the following result for the proof of the theorem.

\begin{lemma}\label{lem:PQ}
Let $P,\,Q\in\Real^{q\times q}$ be symmetric positive semidefinite
matrices satisfying $PQ=0$. Let $\mu\in\Real$ and the nonzero vector
$\eta\in\Complex^{q}$ satisfy
\begin{eqnarray}\label{eqn:PQ}
(P-Q)\eta=\mu\eta.
\end{eqnarray}
The following hold.
\begin{enumerate}
\item If $\mu>0$ then $P\eta=\mu\eta$ and $Q\eta=0$.
\item If $\mu<0$ then $P\eta=0$ and $Q\eta=-\mu\eta$.
\item If $\mu=0$ then $P\eta=0$ and $Q\eta=0$.
\end{enumerate}
\end{lemma}

\begin{proof}
{\em Case~1: $\mu>0$.} Note that $PQ=0$ implies $QP=0$ because the
matrices $P,\,Q$ are symmetric. Multiplying both sides of
\eqref{eqn:PQ} by $-Q$ we obtain $-\mu
Q\eta=-QP\eta+Q^{2}\eta=Q(Q\eta)$ which tells us that the vector
$Q\eta$ if nonzero must be an eigenvector of $Q$ with the negative
eigenvalue $-\mu$. But since $Q\geq 0$ all its eigenvalues must be
nonnegative. Hence $Q\eta=0$. Then \eqref{eqn:PQ} gives us
$P\eta=\mu\eta$. {\em Case~2: $\mu<0$.} Negating \eqref{eqn:PQ} we
can write $(Q-P)\eta=(-\mu)\eta$. The result then follows from the
previous case. {\em Case~3: $\mu=0$.} This time \eqref{eqn:PQ}
implies $P\eta=Q\eta$. Multiplying both sides with $P$ yields
$P^{2}\eta=PQ\eta=0$ thanks to $PQ=0$. Then we can proceed as
follows $0=\eta^{*}P^{2}\eta=\|P\eta\|^{2}$ because $P$ is symmetric
positive semidefinite. And $\|P\eta\|=0$ means $P\eta=0$. Then
$Q\eta=0$ follows by $Q\eta=P\eta$.
\end{proof}

\vspace{0.12in}

\noindent{\bf Proof of Theorem~\ref{thm:MK}.} Let $\M$ and $\K$ be
edge-isolated. This implies that the product of their Laplacians
vanish, i.e., $MK=0$. This is obvious if either ${\rm ver}\,\E_{\rm
m}$ or ${\rm ver}\,\E_{\rm k}$ is empty because an empty edge set
means a zero Laplacian matrix. As for the case that both edge sets
are nonempty we can always label the vertices such that ${\rm
ver}\,\E_{\rm m}=\{v_{1},\,v_{2},\,\ldots,\,v_{r}\}$ and ${\rm
ver}\,\E_{\rm k}=\{v_{s},\,v_{s+1},\,\ldots,\,v_{q}\}$ for some
indices $2\leq r<s\leq q-1$. The corresponding ($q$-by-$q$)
Laplacians then enjoy the block diagonal form
\begin{eqnarray*}
M=\left[\begin{array}{cc}M_{1}&0\\0&0\end{array}\right]\quad
\mbox{and}\quad
K=\left[\begin{array}{cc}0&0\\0&K_{2}\end{array}\right]
\end{eqnarray*}
with $M_{1}\in\Real^{r\times r}$ and
$K_{2}\in\Real^{(q-s+1)\times(q-s+1)}$ which makes it clear that
$MK=KM=0$. Let us introduce the shorthand notation
$\Gamma=B+j(K-M)$. The matrix $\Gamma$ comes with the properties
$\Gamma\one_{q}=0$ and ${\rm Re}\,\lambda_{i}(\Gamma)\geq 0$ for all
$i$. (The demonstration of these properties is very similar to the
demonstration of the same properties satisfied by the matrix
$\Lambda$; see the proof of Theorem~\ref{thm:test3}.) Hence, without
loss of generality, we let $\lambda_{1}(\Gamma)=0$.

Suppose $\lambda_{2}(\Gamma)\leq 0$. This means
$\lambda_{2}(\Gamma)=j\mu$ for some $\mu\in\Real$ because ${\rm
Re}\,\lambda_{i}(\Gamma)\geq 0$ for all $i$. Then we can find an
eigenvector $\xi\notin{\rm span}\,\{\one_{q}\}$ satisfying
$\Gamma\xi=j\mu\xi$ (see the proof of Theorem~\ref{thm:test3}).
Without loss of generality let $\xi$ be a unit vector. By writing
\begin{eqnarray*}
j\mu=\xi^{*}\Gamma\xi=\xi^{*}B\xi+j(\xi^{*}K\xi-\xi^{*}M\xi)
\end{eqnarray*}
we see at once that $\xi^{*}B\xi=0$ because $M,\,B,\,K$ are
symmetric positive semidefinite matrices. Then follows
\begin{eqnarray}\label{eqn:Bxi2}
B\xi=0
\end{eqnarray}
under which $\Gamma\xi=j\mu\xi$ reduces to
\begin{eqnarray}\label{eqn:concern}
(K-M)\xi=\mu\xi\,.
\end{eqnarray}
Let us now study \eqref{eqn:concern} under all three possibilities.
{\em Case~1: $\mu>0$.} By Lemma~\ref{lem:PQ} we have $K\xi=\mu\xi$
and $M\xi=0$. Choosing $\omega=\sqrt{(\mu+k_{0})/m_{0}}$ we can
therefore write
\begin{eqnarray}\label{eqn:Kaw2Ma2}
((K+k_{0}I)-\omega^{2}(M+m_{0}I))\xi=0\,.
\end{eqnarray}
{\em Case~2: $\mu<0$.} By Lemma~\ref{lem:PQ} we have $K\xi=0$ and
$M\xi=-\mu\xi$. This time choosing
$\omega=\sqrt{k_{0}/(-\mu+m_{0})}$ we can establish
\eqref{eqn:Kaw2Ma2}. {\em Case~3: $\mu=0$.} By Lemma~\ref{lem:PQ} we
have $K\xi=M\xi=0$ and \eqref{eqn:Kaw2Ma2} holds with
$\omega=\sqrt{k_{0}/m_{0}}$. Hence for all cases \eqref{eqn:Bxi2}
and \eqref{eqn:Kaw2Ma2} simultaneously hold. Lemma~\ref{lem:one}
then tells us that the oscillators~\eqref{eqn:inertial} do not
synchronize.

To show the other direction suppose the
oscillators~\eqref{eqn:inertial} do not synchronize. By
Lemma~\ref{lem:one} there exist $\omega>0$ and $\xi\notin{\rm
span}\,\{\one_{q}\}$ such that \eqref{eqn:Bxi2} and
\eqref{eqn:Kaw2Ma2} hold. Let us rewrite \eqref{eqn:Kaw2Ma2} as
\begin{eqnarray}\label{eqn:concern2}
(K-\omega^{2}M)\xi=(\omega^{2}m_{0}-k_{0})\xi\,.
\end{eqnarray}
There are three possibilities concerning \eqref{eqn:concern2}. {\em
Case~1: $\omega^{2}m_{0}-k_{0}>0$.} Lemma~\ref{lem:PQ} allows us
write $K\xi=\mu\xi$ and $M\xi=0$ with $\mu=\omega^{2}m_{0}-k_{0}$.
Combining this with \eqref{eqn:Bxi2} we obtain
\begin{eqnarray}\label{eqn:obtain}
\Gamma\xi=j\mu\xi\,.
\end{eqnarray}
{\em Case~2: $\omega^{2}m_{0}-k_{0}<0$.} By Lemma~\ref{lem:PQ} this
time we have $K\xi=0$ and $M\xi=-\mu\xi$ with
$\mu=m_{0}-k_{0}/\omega^{2}$ and again \eqref{eqn:obtain} follows.
{\em Case~3: $\omega^{2}m_{0}-k_{0}=0$.} In this final case we have
to have $K\xi=0$ and $M\xi=0$ according to Lemma~\ref{lem:PQ}. Then
\eqref{eqn:obtain} holds with $\mu=0$. Now, in the light of
$\xi\notin{\rm span}\,\{\one_{q}\}$ and $\Gamma\one_{q}=0$ we can
deduce from \eqref{eqn:obtain} that $\Gamma$ has at least two
eigenvalues on the imaginary axis. Combining this with the fact that
all the eigenvalues of $\Gamma$ are on the closed right half-plane
we reach the conclusion ${\rm Re}\,\lambda_{2}(\Gamma)=0$. Hence the
result.\hfill\null\hfill$\blacksquare$

\section{Conclusion}

In this paper we studied the collective behavior of harmonic
oscillators that are communicating through inertial, dissipative,
and restorative connectors. The coupling considered was fixed and
symmetric. We showed that whether the oscillators tend to
synchronize or not can be determined through the spectrum of a
single complex Laplacian matrix, which is constructed from the three
individual Laplacians, each representing a different type of
coupling. We also provided certain structural conditions (on the
coupling graphs) which render the Laplacian construction much
simpler. The theorems presented here generalize some earlier
results.

\bibliographystyle{plain}
\bibliography{references}

\begin{thebibliography}{10}

\bibitem{cai10}
C.~Cai and S.E. Tuna.
\newblock Synchronization of nonlinearly coupled harmonic oscillators.
\newblock In {\em Proc. of the American Control Conference}, pages 1767--1771,
  2010.

\bibitem{diestel97}
R.~Diestel.
\newblock {\em Graph Theory}.
\newblock Springer-Verlag, 1997.

\bibitem{liu19}
J.~Liu, H.~Li, and J.~Luo.
\newblock Impulse bipartite consensus control for coupled harmonic oscillators
  under a coopetitive network topology using only position states.
\newblock {\em IEEE Access}, 7:20316--20324, 2019.

\bibitem{ren08}
W.~Ren.
\newblock Synchronization of coupled harmonic oscillators with local
  interaction.
\newblock {\em Automatica}, 44:3195--3200, 2008.

\bibitem{song19b}
Q.~Song, F.~Liu, J.~Cao, A.V. Vasilakos, and Y.~Tang.
\newblock Leader-following synchronization of coupled homogeneous and
  heterogeneous harmonic oscillators based on relative position measurements.
\newblock {\em IEEE Transactions on Control of Network Systems}, 6:13--23,
  2019.

\bibitem{song16}
Q.~Song, F.~Liu, G.~Wen, J.~Cao, and Y.~Tang.
\newblock Synchronization of coupled harmonic oscillators via sampled position
  data control.
\newblock {\em IEEE Transactions on Circuits and Systems I: Regular Papers},
  63:1079--1088, 2016.

\bibitem{song19}
Q.~Song, G.~Lu, G.~Wen, J.~Cao, and F.~Liu.
\newblock Bipartite synchronization and convergence analysis for network of
  harmonic oscillator systems with signed graph and time delay.
\newblock {\em IEEE Transactions on Circuits and Systems I: Regular Papers},
  66:2723--2734, 2019.

\bibitem{su09}
H.~Su, X.~Wang, and Z.~Lin.
\newblock Synchronization of coupled harmonic oscillators in a dynamic
  proximity network.
\newblock {\em Automatica}, 45:2286--2291, 2009.

\bibitem{sun14}
W.~Sun, J.~Lu, S.~Chen, and X.Yu.
\newblock Synchronisation of directed coupled harmonic oscillators with
  sampled-data.
\newblock {\em IET Control Theory and Applications}, 8:937--947, 2014.

\bibitem{tuna17}
S.E. Tuna.
\newblock Synchronization of harmonic oscillators under restorative coupling
  with applications in electrical networks.
\newblock {\em Automatica}, 75:236--243, 2017.

\bibitem{tuna19}
S.E. Tuna.
\newblock Synchronization of small oscillations.
\newblock {\em Automatica}, 107:154--161, 2019.

\bibitem{wang19}
G.~Wang, J.~Ji, and J.~Zhou.
\newblock Practical stochastic synchronisation of coupled harmonic oscillators
  subjected to heterogeneous noises and its applications to electrical systems.
\newblock {\em IET Control Theory \& Applications}, 13:96--105, 2019.

\bibitem{zhang18}
H.~Zhang and J.~Ji.
\newblock Group synchronization of coupled harmonic oscillators without
  velocity measurements.
\newblock {\em Nonlinear Dynamics}, 91:2773--2788, 2018.

\bibitem{zhang12}
H.~Zhang and J.~Zhou.
\newblock Synchronization of sampled-data coupled harmonic oscillators with
  control inputs missing.
\newblock {\em Systems \& Control Letters}, 61:1277--1285, 2012.

\bibitem{zhang13}
Y.~Zhang, Y.~Yang, and Y.~Zhao.
\newblock Finite-time consensus tracking for harmonic oscillators using both
  state feedback control and output feedback control.
\newblock {\em International Journal of Robust and Nonlinear Control},
  23:878--893, 2013.

\bibitem{zhou12}
J.~Zhou, H.~Zhang, L.~Xiang, and Q.~Wu.
\newblock Synchronization of coupled harmonic oscillators with local
  instantaneous interaction.
\newblock {\em Automatica}, 48:1715--1721, 2012.

\end{thebibliography}
\end{document}